\documentclass[11pt]{amsart}
\usepackage{latexsym,amssymb,amsthm,amsmath,amscd}

\theoremstyle{plain}
\newtheorem{theorem}{Theorem}[section]
\newtheorem*{Theorem B}{Theorem B}
\newtheorem*{Theorem A}{Theorem A}

\newtheorem{lemma}{Lemma}[section]

\newtheorem{proposition}{Proposition}[section]

\newtheorem{definition}{Definition}[section]
\newtheorem{example}{Example}[section]
\numberwithin{equation}{section}

\theoremstyle{remark}
\newtheorem{remark}{Remark}[section]

 \numberwithin{equation}{section}

\def\<{\left < }
\def\>{\right >}
\def\({\left ( }
\def\){\right )}

\def\e{\eqref}

\begin{document}

\title[Totally geodesic submanifolds of symmetric spaces, III]{Totally geodesic submanifolds of symmetric spaces, III}

\author[B.-Y. Chen, P.-F. Leung and T. Nagano]{Bang-Yen Chen, Pui-Fai Leung and Tadashi Nagano}

\begin{abstract} One purpose of this article is to establish a general method to determine stability of totally geodesic  submanifolds of symmetric spaces. The method is used to determine the stability of the basic totally geodesic submanifolds $M_+,M_-$ introduced and studied by Chen and  Nagano in [Totally geodesic submanifolds of symmetric spaces, II, {Duke Math. J.}  45 (1978), 405--425] as minimal submanifolds. The other purpose  is to establish a stability theorem for minimal totally real submanifolds of K\"ahlerian manifolds.

\end{abstract}


\maketitle

\section{Introduction}

One purpose of this paper is to establish a general method to determine stability of totally geodesic  submanifolds of symmetric spaces. The method will be used to determine the stability of the basic totally geodesic submanifolds $M_+,M_-$ introduced and studied in part II of this series \cite{CN2} as minimal submanifolds. The other purpose of this paper is to establish a stability theorem for minimal totally real submanifolds of K\"ahlerian manifolds.

The problem of  stability amount to finding a good estimate of the smaller eigenvalues of a certain elliptic linear differential operator $L$, which is formi\-dably  difficult in general. However, in the presence of an abundant group action leaving $L$ invariant, $L$ can be related to a certain Casimir operator and the problem can be resolved by means of the representation theory. This idea will be explained in section 2 to establish the main theorem (Theorem 2.1), which some of its applications are derived in section 3, We determine the  stability of $M_+$ and $M_-$ in section 4.  Stability of minimal totally real submanifolds are studied in section 5.

As to earlier work, we mention Fomenko (\cite{Fo} and other papers) who determined compacta or certain currents which realize the absolute minimum of the volume within their cobordism or homology classes and the complete determination of stable minimal submanifolds of the complex projective space by Lawson-Simons \cite{LS}.

\section{Stability of minimal submanifolds}

Let $f:N \rightarrow M$ be an immersion from a compact
$n$-dimensional manifold $N$ into an $m$-dimensional
Riemannian manifold $M$. Let $\{f_{t}\}$ be a 1-parameter
family of immersions of $N \rightarrow M$ with the
property that $f_{0}=f$. Assume the map $F: N\times
[0,1] \rightarrow M$ defined by $F(p,t)=f_{t}(p)$ is
differentiable. Then $\{f_{t}\}$ is called a {\it
variation\/} of $f$. A variation of
$f$ induces a vector field in $M$ defined along
the image of $N$ under $f$. We shall denote this field by
$\zeta$ and it is constructed as follows:

Let $\partial/\partial t$ be the standard vector field in $N\times [0,1]$. We set
$$\zeta(p)=F_{*}\!\(\text{$\small\frac{\partial}{\partial t}$}(p,0)\!\).$$
Then $\zeta$ gives rise to cross-sections $\zeta^T$ and $\zeta^N$ in $TN$ and $T^{\perp}N$, respectively. If we have $\zeta^{T}=0$, then $\{f_{t}\}$ is called a {\it normal variation\/} of $f$. For a given normal vector field $\xi$ on $N$, $exp\,t\xi$ defines a normal variation $\{f_{t}\}$ induced from $\xi$. We denote by ${\mathcal V}(t)$ the volume of $N$ under $f_t$ with respect to the induced metric and by ${\mathcal V}'(\xi)$ and ${\mathcal V}''(\xi)$, respectively, the values of the first and the second derivatives of ${\mathcal V}(t)$ with respect to $t$, evaluated at $t=0$. 

The following formula is well-known:
$${\mathcal V}'(\xi)=-n\int_{N} \<\xi,H\> *1.$$ 

A  compact minimal submanifold $N$ in a Riemannian manifold $M$ is called {\it stable} if the second variation of the the volume integral is nonnegative for every normal vector field $v$. The second variation  formula is well-known:
\begin{align}\label{2.1} {\mathcal
V}''(\xi)=\int_N \{||\nabla v||^2 -R(v,v)-||Bv||^2\}*1,\end{align}
where $\nabla$ is the connection of the normal bundle $T^\perp N$, $R(v,v)$ is defined to be $\sum_i\<R^M(e_i,v)v,e_i\>$ at a point $x$ for any orthonormal basis $(e_i)$ for the tangent space $T_xN$ to $N$ at $x$ with the curvature tensor $R^M$ of $M$, and $B$ is the second fundamental form of $N$ in $M$; thus $Bv$ is a $T^\perp N$-valued 1-form. 

Applying the Stokes theorem to the integral of the first terms (as Simons \cite{S} did), we have
\begin{align}\label{2.2}{\mathcal V}''(\xi)= \int_N \<Lv,v\>*1,\end{align}
 in which $L$ is a self-adjoint,  strongly elliptic linear differential operator of the second order acting on the space of sections of the normal bundle given by
$$L=-\Delta^{D}-{\hat A}-Q,$$
where $\Delta^D$ is the Laplacian operator associated with the normal connection, $\<\right.\! {\hat  A}\xi,\eta\!\left.\>=trace\<\right.\! A_{\xi},A_{\eta}\!\left.\>$, and $\<\right.\! Q\xi,\eta\!\left.\>=R(\xi,\eta).$

The differential operator $L$ is called the {\it Jacobi operator\/} of $N$ in $M$. The differential operator $L$ has discrete eigenvalues $\lambda_{1} < \lambda_{2} <\ldots
\,\nearrow \infty.$ We put $$E_{\lambda}=\{\xi\in \Gamma(T^{\perp}N)\, :\, L(\xi)=\lambda\xi \,\}.$$ The number of $\sum_{\lambda <0}dim(E_{\lambda})$ is called
the {\it index\/} of $N$ in $M$. A vector field $\xi$ in $E_0$ is called a {\it Jacobi field.\/} 
Thus, $N$ is stable if and only if the eigenvalues of $L$ are all nonnegative. It is obvious but important that \e{2.1} (= \e{2.2}) is zero if $v$ is the restriction of a Killing vector field on $M$. We will show below the existence of such a $v$ when $N$ is a totally geodesic submanifold of a symmetric space $M$, which we assume from now on unless mentioned otherwise. Now that $B=0$, stability obtains trivially when $R$ is nonnegative. For this reason we are interested in the case $M$ is compact. Also, we assume that $M$ is irreducible partly to preclude tori as $M$, although this assumption is not trivial (see Remark \ref{R:2.1}).

We need to fix some notations. Since $N$ is totally geodesic, there is a finitely covering group $G_N$ of the connected isometry group $G_{N}^o$ of $N$ such that $G_N$ is a subgroup of the connected isometry group $G_M$ of $M$ and leaves $N$ invariant, provided that $G_{N}^o$ is semi-simple. Let $\mathcal P$ denote the orthogonal complement of the Lie algebra ${\it g}_{_{N}}$ in the Lie algebra ${\it g}_{_{M}}$ with respect to the bi-invariant inner product on ${\it g}_{_{M}}$ which is compatible with the metric of $M$. Every member of ${\it g}_{_{M}}$  is though of as a Killing vector field because of the action $\, G_M$ on $\, M$. 

Let $\hat P$ denote the space of the vector fields corresponding to the member of $\mathcal P$
restricted to the submanifold $N$.

\vskip.1in
\begin{lemma}\label{L:2.1}  To every member of $\, \mathcal P$ there corresponds a unique (but not canonical) vector field $v \in {\hat P}, v$ is a normal vector field and hence ${\hat P}$ is a $G_N$-invariant subspace of the space $\Gamma (T^{\perp}N)$ of  sections of the normal bundle to $N$. Moreover, ${\hat P}$ is homomorphic with $\mathcal P$ as a $G_N$-module.
\end{lemma}
\begin{proof} Let $o$ be an arbitrary point of $N$. Let $K_M$ and $K_N$ denote the isotropy subgroups of $G_M$ and $G_N$ at $o$, respectively. Then  ${\it g}_{_{M}}/
{\it k}_{_{M}}$ and ${\it g}_{_{N}}/{\it k}_{_{N}}$ and ${\mathcal P}/({\mathcal P} \cap {\it
k}_{_{M}})$ are identified with $T_{o}M$, $T_{o}N$ and  $T_{o}^{\perp}N$ by isomorphisms induced by the evaluation of vector fields in ${\it g}_{_{M}}$ at $o$. In particular, the value $v(o)$ of $v$ is normal to $N$.
This proves the lemma. \end{proof}

Now we are ready to explain our method for determining stability. The group $G_N$ acts on sections in
$\Gamma(T^{\perp}N)$ and hence on the differential operators: $\Gamma(T^{\perp}N) \rightarrow
\Gamma(T^{\perp}N)$. $G_N$ leaves $L$ fixed since $L$ is defined with $N$ and the metric of $M$ only. Therefore, each eigenspace of $L$ is left invariant by $G_N$. Let $V$ be one of its $G_N$-invariant irreducible subspaces. We have  a representation $\rho : G_{N} \rightarrow GL(V).$ We denote by $c(V)$ or $c(\rho)$ the eigenvalue of the corresponding  Casimir operator, which we will explain shortly. Then Theorem \ref{T:2.1} say, modulo details, that $N$ is stable if and only if $c(V)\geq c(\mathcal P)$ for every such $V$.

 To define $c(V)$ we fix an orthonormal basis $(e_{\lambda})$ for ${\it g}_{_{N}}$ and consider the linear endomorphism $C$ or $C_V$ of $V$ defined by 
\begin{align}\label{2.3}C=-\sum\,\rho(e_{\lambda})^{2}.\end{align}
It is known that $C$ is $c(V)I_V$ (see Chapter 8 of \cite{B}), where $I_V$ is the identity map on $V$. In our case, the {Casimir operator} $$C_{V}=-\sum [e_{\lambda},[e_{\lambda},V]\,]$$ for every member $v$ of $V$ (after extending to a neighborhood of $N$).

Accepting the theorem for the moment, we have an {\it algorithm for stability} goes like this. One can compute $c(V)$ by the Freudenthal formula (cf. \cite[Chapter 8, page 120]{B}) once one knows the action $\rho$ of $G_N$ on $V$. So, the rest is to know all the simple $G_N$-modules $V$ in $\Gamma(T^\perp N)$. This is done by means of the Frobenius theorem as reformulated by R. Bott, which asserts in our case, that a simple $G_N$-module $V$ appears in $\Gamma(T^\perp N)$ if and only if $V$ as a $K_N$-module contains a simple $K_N$-module which is isometric with a $K_N$-module of $T^\perp_o N$.

\begin{theorem}\label{T:2.1}  A compact, connected, totally geodesic submanifold $N= G_{N}/K_{N}$ of a locally symmetric space $M=G_{M}/K_{M}$ is stable as a minimal
submanifold if and only if one has $c(V)\geq c(P')$ for the eigenvalue of the Casimir operator of every simple $G_N$-module $V$ which shares as a $K_N$-module some simple $K_N$-submodule of the $K_N$-module $T_{o}^{\perp}N$ in common with some simple $G_N$-submodule $P'$ of $P$. 
\end{theorem}
\begin{proof} We first prove the next (a) through (c) by using local expressions of $L$ and $C$.
\vskip.05in

{(a)} {\it The difference $\, L-C \,$ is an operator of order zero.}
\vskip.05in

 {(b)} {\it The difference $\, L-C \,$ is given by a self-adjoint endomorphism $Q$ of the normal bundle $T^{\perp}N$. }
\vskip.05in

 {(c)} {\it The endomorphism $Q$ is $G_N$-invariant.}
\vskip.1in 

Given a point $x$ of $N$, we choose a basis of ${\it g}_{_{N}}$ given by 
$$(e_{\lambda})=(\ldots,e_{i},\ldots,e_{\alpha},\ldots) $$ 
and a finite system $(e_{r})$ of vectors in ${\mathcal P} \subset {\it g}_{_{M}}$ such that 

(1) $(e_{i}(p))_{1\leq i \leq n}$ is an orthonormal basis for the tangent space $T_{p}N$, 

(2) ${\nabla e_{i}}=0,\, 1\leq i \leq n,$ at $x$, 

(3) $e_{\alpha}(x)=0,\, n < \alpha \leq dim\,{\it g}_{_{N}}$, and 

(4) $(e_{r}(x))$ is an orthonormal basis for the normal space $T_{x}^{\perp}N$, 

\noindent which we can do as is well-known.

We write $\nabla$ for the connection $\nabla^M$ of $M$, since we know $\nabla^M=\nabla$ on the tangent and the normal vector fields to $N$. 

An arbitrary normal vector field $v$ is written as $v=\sum v^{r}e_r$ on a neighborhood of $x$ by Lemma \ref{L:2.1}. Evaluating $L\xi$ and $C\xi$ at $x$, we obtain 
$$Lv=-\sum\,{\nabla}_{e_{i}}{ \nabla}_{e_{i}}v -\sum\,R_{rs}v^{r}e_{s},$$ 
where $R_{rs}$ are the components of $R$ in \e{2.1} and 
$$Cv= -\sum \,[e_{\lambda},[e_{\lambda},v]\,]$$
$$=-\sum\,{ \nabla}_{e_{i}}{\nabla}_{e_{i}}v+\sum\, v^{r}{ \nabla}_{e_{i}}{ \nabla}_{e_{r}}e_{i} -\sum \,v^{r}({ \nabla}_{e_{r}}e_{\alpha})^{s}{ \nabla}_{e_{s}} e_{\alpha}$$
$$=-\sum \,{ \nabla}_{e_{i}}{\nabla}_{e_{i}}v  -\sum\,v^{r}({ \nabla}_{e_{r}}e_{\alpha})^{s} { \nabla}_{e_{s}}e_{\alpha},$$ 
where for the vanishing of the second term we use the fact that $e_i$ is a Killing vector
field. Thus we find 
$$(L-C)v=\sum\, v^{r}({ \nabla}_{e_{r}}e_{\alpha})^{s}\,{ \nabla}_{e_{s}}
e_{\alpha} -\sum\, R_{rs}v^{r}e_{s}$$ 
$$=\sum\, (A_{\alpha})^{2}v-R(v, *),$$ where   $A_{\alpha}$ is the Weingarten map given by the restriction of the operator: $X \rightarrow-{ \nabla}_{X}e_{\alpha}$ on $T_{x}N$ to the normal space $T_{x}^{\perp}N$. This proves  statement (a). Since $A_\alpha$ is skew-symmetric, $(A_\alpha)^2$ is symmetric. Statement (c) is obvious from the $G_N$-invariance of $L$ and $C$.
 
The theorem follows from statements (a), (b), (c) easily when the isotropy subgroup $K_N$ is irreducible on the normal space $T_{x}^{\perp}N \cong {\it g}_{_{M}}/({\it g}_{_{N}} \oplus {\it k}_{_{M}})$. In fact, $Q$ is then a constant scalar multiple of the identity map of $T^{\perp}N$; $Q=k\cdot I,$ by statements (a) through (c) and Schur's lemma. $N$ is stable if and only if the eigenvalues of $L$ are all non-negative. Since $L=(c(P')+k)\cdot I=0$ on the normal Killing fields (see Lemma \ref{L:2.1}), this is equivalent to say that $0 \leq c(V)+k=c(V)-c(P')$ for every simple $G_N$-module $V$ in $\Gamma (T^{\perp}N)$ (which is necessarily contained in
an eigenspace of $L$ by $Q  = k\cdot I),$ and the Bott-Frobenius theorem completes the proof.

In the general case, we decompose the normal space into the direct sum of simple $K_N$-modules: $\mathcal P'\oplus \mathcal  P''\oplus \cdots \,.$ Accordingly, we have $T^{\perp}N=E'\oplus E'' \oplus \ldots,$ where $E', E'',\ldots,$ etc. are obtained from $P', P'',\ldots,$ etc., in the usual way by applying the action of $G_N$ to the vectors in $\mathcal P', \mathcal P'',\ldots,$ etc. Since $G_N$ leaves invariant $E', E'',\ldots,$ the normal connection leaves invariant the section spaces $\Gamma (E'), \Gamma (E''),\ldots \, .$ Hence, $L$ and $C$ leave these spaces invariant. (For this, the irreducible subspaces $\mathcal P',\mathcal  P'',
\ldots$ must be taken within eigenspaces of the symmetric operator ${\bar S}$ at the point o). In particular, the projections of $\Gamma (T^{\perp}N)$ onto $E', E'',\ldots,$ etc. commute with $C$ and $L$. Thus one can repeat the argument for irreducible case to each of $E',E'',\ldots $ to finish the proof of the theorem.
\end{proof}

\begin{remark}\label{R:2.1} In the sequel we will assume that the ambient symmetric space $M$ is irreducible. But the reducible case is far from trivial. Indeed, for instance, the diagonal in the Riemannian product $M\times M$ is stable as a minimal submanifold if and only if the identity map of $M$ is stable as a harmonic map and this condition obtains for some $M$ (e.g. $S^2$) but not for the others (e.g. $S^n$ for $n>2$) (cf. \cite{N} for this and determination of stability for individual symmetric spaces).
\end{remark}

\section{Consequences}

The following proposition, faintly reminiscent of Synge's lemma, is a very simple application of Theorem \ref{T:2.1} and yet most useful the the next section.

\begin{proposition}\label{P:3.1} {\it A compact, connected, totally geodesic submanifold $N$ of a compact symmetric space $M$ is unstable as a minimal submanifold if the normal bundle admits a nonzero $G_N$-invariant section and if the centralizer of $G_N$ in $G_M$ is discrete.}
\end{proposition}
\begin{proof} Let $v$ be a nonzero $G_N$-invariant normal vector field on $N$. We have $\nabla v=0$. In view of \e{2.1} we will show that $R(v,v)$ is positive. 

The sectional curvature of a tangential 2-plane at a point $x\in N$ equals
$||[e,f]||^{2}$ if (i) $e$ is a member of ${\it g}_{_{N}}$, (ii) $f$ is that of ${\it g}_{_{M}}$, (iii)
$e(x)$ and $f(x)$ form an orthonormal basis for the 2-plane, and (iv) ${ \nabla}e={\nabla}f=0$ at $x$. Therefore, $R(v,v)$ fails to be positive only if $[e,f]=0$ for every such $e$ and $f$ satisfying $v(x)\wedge f(x)=0.$ Since the isotropy subgroup $K_N$ at $x$ leaves the normal vector $v(x)$ invariant, we have $[e',f]=0$ for every member $e'$ of ${\it k}_{_{M}}$ and hence $[{\it g}_{_{N}},f]=0$ if $R(v,v)=0$ at $x$. Such an $f$ generates a subgroup in the centralizer of ${\it g}_{_{N}}$ in ${\it g}_{_{M}}$. This contradicts to the assumption. 
\end{proof}

\vskip.1in

\begin{example}\label{E:3.1}
 {\rm Let $N$ be the equator in the sphere $M=S^n$. That $N$ is unstable follows from the proposition  if one considers a unit  normal vector field to it. The centralizer in this case is generated by the antipodal map: $x \rightarrow -x.$ Its orbit space is the real projective space $M'$. The projection: $M \rightarrow M'$ carries $N$ onto a hypersurface $N'$. The reflection in $N'$ is a member of $G_N$ by our general agreement on $G_N$ (if $n>1$) and precludes the existence of non-vanishing $G_N$-invariant normal vector field to $N'$. It is clear by Theorem \ref{T:2.1} that $N'$ is stable.}
 \end{example}

\vskip.1in
\begin{remark} In general, if $N$ is a stable minimal submanifold of a Riemannian manifold $M$ and $M$ is a covering Riemannian manifold of $M'$, then the projection $N'$ of $N$ in $M$ is stable too. The example above shows the converse is false.
\end{remark}

\vskip.1in
\begin{definition}\label{D:3.1} {\rm For a compact connected symmetric space $M=G_M/K_M$, $G_M$ is semisimple,  there is a unique symmetric space $M^*$ of which $M$ and every connected symmetric space which is locally isomorphic with $M$ are covering Riemannian manifold of $M^*$. We call $M^*$ the {\it bottom space} of $M$. If $M$ is a group manifold, $M^*$ is the adjoint group $ad(M)$.
}\end{definition}

By applying Theorem 5.2 and Proposition 5.3 above, we may
obtain the following result.

\begin{proposition}\label{P:3.2} A compact subgroup $N$ of a compact Lie group $M$ is stable with respect to a bi-invariant metric on $M$ if 

{\rm (a)} $N$  has the same rank as $M$ and

{\rm (b)}  $M=M^{*}$, that is, $M$ has no nontrivial center. 
\end{proposition}
\begin{proof} The compact group manifold $M$ has $G_{M}=M_{L}\times M_{R},$ where $M_L$ is the left translations group $M\times \{1\}$ and $M_R$ the right
translation group; here $M_R$ acts ``to the left'' too, that is, $(1,a)$ carries $x$ into $xa^{-1}$. 

Similarly for $G_N$, $G_N$ is effective on every invariant neighborhood of $N$ in $M$ by (b). We first consider the case where $N$ is a maximal toral subgroup $T$ of $M$. Let $A_T$ denote the subgroup $\{(a,a^{-1}):a\in T\}$ of $G_N$. We have an epimorphism $\epsilon : K_{N}\times A_{T} \rightarrow G_T$ by the multiplication whose kernel, $ker\,\epsilon$, is the subgroup of elements of order 2. 

In order to use Theorem \ref{T:2.1}, we look at an arbitrary simple $G_T$-module $V$ in $\Gamma(E')$ where $E'$ is, as before, the vector bundle $G_{N}\mathcal P'$ defined from the simple $K_T$-submodule $\mathcal P'$ of the normal space. $\mathcal P'$ is a root space corresponding to a root $\alpha$ of ${\it g}_{_{M}}$. With $V$ we compare the space $P'$, a simple $K_N$-module in $\Gamma(E')$ which is defined from the members of the Lie algebra of $M_L$ taking values in $\mathcal P'$ at a point of $N$, or equivalently, which is obtained from $\mathcal P'$ by applying $N_R$ to the members of $\mathcal P'$. We want to show $c(V)\geq c(P').$ Since $\alpha\not= 0$ by (a), both $V$ and $P'$ have dimension 2 and these are isomorphic as $K_T$-modules. 

The relationship between $V$ and $P'$ can be made more explicitly. Namely, a basis for $P'$ is a global frame of $E'$ and therefore the sections in $V$ are linear combinations of the basis vectors whose components are functions on $N$. These functions form a simple $G_T$-module $F$ of dimension 2 and $V$ is a $G_T$-submodule of $F\otimes P'$. By the Bott-Frobenius theorem, $K_T$ acts trivially on a 1-dimensional subspace
of $F$. Every weight $\varphi$ of $F$ is a linear combination of roots of ${\it g}_{_{M}}$ whose coefficients are even numbers. In fact all the weights of the representations of $G_T$ are linear combinations of those roots over the integers by (a) and (b) and, since
$K_{T}\cap A_{T} \cong ker\,\epsilon$ is trivial on $F$, the coefficients must be even.

On the other hand, if one looks at the definition of Casimir operator, $C=-\sum\ \rho(e_{\lambda})^2$, one sees that the eigenvalue $c(V)$ is a sort of average of
the eigenvalues of $-\rho(x)^2$, $||x||=1,$ or more precisely, $$c(V)=-\frac{1}{\dim\,V}\int_{||x||=1} trace(\rho(x))^2,$$ where the integral is taking over the unit sphere of the Lie algebra with an appropriately normalized invariant measure. For this reason, showing $c(V)\geq c(P')$, or equivalently, $(\varphi+\alpha)^{2}-\alpha^{2}\geq 0$
amounts to showing the inner product $$\<2\alpha + \varphi ,\alpha\>=\<\varphi+\alpha,\varphi+\alpha\>-\<\alpha,\alpha\>\,\,\geq 0$$ in which we may assume that $\varphi$ is dominant (with respect to the Weyl group of ${\it g}_{_{M}}$). And this concludes the proof for  $N=T$.

We turn to the general case $N\supset T.$ We assume $N$ is unstable and will show this contradicts the stability of $T$. There is then a simple $G_N$-module $V$ of normal vector fields to $N$ such that 
the second variation \e{2.1} is negative for some member $v$ of $V$. If we restrict $v$ to $T$ we still have a normal vector field but the integrand in \e{2.1} for $v_{|\,T}$ will differ from the restriction of the integrand for $v$ by the terms corresponding to the
tangential directions to $N$ which are normal to $T$. However, a remedy comes from the group action. First, \e{2.1}  is invariant under $G_N$ acting on $V$.
Second, every tangent vector to $N$ is carried into a tangent vector to $T$ by some isometry in $G_N$. 
Third, $N$ and $T$ are totally geodesic in $G$ (so $B=0$ in \e{2.1}, but more
importantly the connection and the curvature restrict to the submanifolds comfortably). And finally, the isotropy subgroup $K_N$ acts irreducibly on the tangent space to
each simple or circle normal subgroup of $G_N$. From all these it follows that \e{2.1} for $v$ is a positive constant multiple of \e{2.1} for $v_{|\,T}$, as one sees by integrating \e{2.1} for $g(v)_{|\,g(T)},\,g\in G$, over the group $G$ and over the unit sphere of $V$. \end{proof}

\begin{remark} Neither the assumption (a) or (b) can be omitted from Proposition \ref{P:3.2} as the examples of $M=SU(2)$ with $N=SO(2)$ and $M=G_2$ with $N=SO(2)$ show. Also the Proposition will be false if $M$ is not a group manifold, a counterexample being $M=M^{*}=GI$ with $N=S^{2}\cdot S^2$ (local product) among a few others to be explained in the next section. \end{remark}

\section{Stability of basic totally geodesic submanifolds}

By the ``{\it basic}'' totally geodesic submanifolds of a compact symmetric space $M$, we mean the submanifolds $M_+$ and $M_-$ introduced and studied in \cite{CN2}, which may defined as follows. Fix a point $o$ of $M$ and consider the symmetry $s_o$ of $M$ at $o$. Then $M_+$ is an arbitrary connected component $\ne \{o\}$ of the fixed point set $F(s_o,M)$ of $s_o$. And $M_-$, for each $M_+$ and a point $p\in M_+$, is the connected component through $p$ of $F(s_o s_p,M)$. The set of the isomorphism classes of the pairs $(M_+,M_-)$ is independent of $o$ and $p$ and determines $M$ completely.

In this section, $M=G_M/K_M$ is assumed to be 
irreducible and $M$ is the bottoms space $M^*$ defined in Definition \ref{D:3.1}.

\begin{proposition}\label{P:4.1} We have the following.

{\rm (a)} Among the compact connected
simple Lie groups $M^{*}$, the only ones that have
unstable $M_{+}^*$ are $SU(n)^*$, $SO(2n)^*$ with n odd,
$E_{6}^*$ and $G_2$. 

{\rm (b)}  The unstable $M_+$ are $G^{\mathbb C}(k,n-k),\, 0<k<n-k,$
for $SU(n)^*$; $SO(2n)/U(n)^*$ for $SO(2n)^*$; $EIII^*$
for $E_{6}^*$; and $M_{+}^*$ for $G_2$.

{\rm (c)}  Every $M_-$ is stable for the group $M^*$.\end{proposition}

{\bf Comments on the Proof.} 

(I) The stability of $M_-$ is immediate from Proposition \ref{P:3.2} since $M_-$ has the same rank as $M$ (see \cite{CN2}). Otherwise the proof is based on scrutinizing all the individual cases and omitted except for a few cases to illustrate our methods.

(II) Take $M^{*}=SO(2n+1)$. Then $M_{+}= G^{\mathbb R}(k,2n+1-k),\, 0<k<n-k,$ the Grassmannians of the unoriented $k$-planes in $\mathbb E^{2n+1}$ by Table I in \cite{CN2}. The action of $G_{+}=SO(2n+1)$ on $P$ (in the notation of Lemma \ref{L:2.1}) is the adjoint representation corresponding to the highest weight ${\tilde \omega}_2$ in Bourbaki's notation \cite{B}. By Freudenthal's formula, one finds that ${\tilde\omega}_1$ is the only representation that has a smaller eigenvalue than ${\tilde\omega}_{2}; c({\tilde \omega}_{1})<c({\tilde\omega}_{2}).$ But ${\tilde\omega}_1$ does not meet the Bott-Frobenius condition simply because its dimension $2n+1$ is too small. Therefore, $M_+$ is stable by Theorem \ref{T:2.1}. 

(III) While tables in \cite{CN2} describe local types of $M_+$ and $M_-$, we actually need their global types. In most cases the following method is enough for this. The Lefschetz number of $s_o$ for $M^*$ is clearly $2^r$, where $r$ is the rank of $M^*$. This equals the sum of the Euler number ${\mathcal X}(M_+)$ of all $M_+$ and that of $\{o\}$. Each ${\mathcal X}(M_+)$ is positive by Corollary 3.7 of \cite{CN2}. In the above case (II), one checks $2^n=\sum_p \binom{n}{p}$ and concludes the given $M_+$ are the right ones. In some cases which appear in the next proposition, one has to use the homotopy group $\pi_2(M)$ which was determined by Takeuchi \cite{T}.

(IV) Take $SU(n)^*$ for another example. We know $M_{+}=G_{+}/K_{+}=G^{\mathbb C}(k,n-k),$ the complex Grassmann manifold. If $k\not= n-k,\, M_+$ is simply-connected and hence $K_+$ is connected. 

On the other hand, $M_{-}=K_{+}=S(U(k)\times U(n-k)),$ which contains a circle group as the center. Therefore, $M_+$ admits a unit $G_+$-invariant normal vector field. Moreover, the centralizer  of $G_+$ in $G_M$ is trivial. Hence
Proposition \ref{P:3.1} applies to conclude that $M_+$ is unstable. This argument fails in the case $k=n-k$ and we can conclude the stability of $M_+$ by Theorem \ref{T:2.1}  as in (I). 

(V) Instability is established by means of Proposition \ref{P:3.1} except for the case of $G_2$. In this case we have $c({\tilde\omega}_{1})<c({\tilde\omega}_{2})=c$ (the adjoint representation). This ${\tilde\omega}_1$ gives a monomorphism of $G_2$ into $SO(7)$ which restricts to a monomorphism of $K_{+}=SO(4)$ into $SO(4)\times SO(3)$ in $SO(7)$ and then projects to $SO(3)$. This implies that ${\tilde\omega}_1$ appears in a space of normal vector fields.
  
\begin{proposition}\label{P:4.2} Let $M^*$ be a compact symmetric space $G/K$ with $G$ simple. Then, among the $M_+$ and $M_-$, the unstable minimal submanifolds are $G^{\mathbb R}(k,n-k),\,k<n-k,$ in $AI(n)^*$; $G^{\mathbb H}(k,n-k),\, k<n-k,$ in $AII(n)^*$; $SO(k)$  in $G^{\mathbb R}(k,k)$ with k odd; $M_{+} = M_{-}=SO(2)\times AI(n)$ in $CI(n)^*$;
$M_{+}=M_{-}=SO(2)\times AII({n\over 2})$ in $DIII^{*}=
SO(2n)/U(n)$ with n even; $G^{\mathbb H}(2,2)$ in $EI^*$; $FII$ in
$EIV^*$; $AII(4)$ in $EV^*$; and $M_{+}=M_{-}= S^{2}\cdot S^{2}$ in $GI$. \end{proposition} 

{\bf Comments on the Proof.} 

(I) In some cases, one can use another method to get the results quickly. For instance, if $M^*$ is K\"ahlerian, then it is well-known that every compact complex submanifold is stable (Federer \cite{F}). 

(II) Mostly, instability is established by using Proposition \ref{P:3.1}. In the cases, $M_{+}=M_{-}=SO(2)\times L$, this proposition does not literally apply but instability is proven in the same spirit. Consider, say $SO(2)\times AI(n)$ in $CI(n)^*$. This space in $M^*$ is $U(n)/O(n)$. The normal space is isomorphic with the space of the symmetric bilinear forms on $\mathbb E^n$ as an $O(n)$-module. Therefore, there is a $U(n)$-invariant unit normal vector field $v$ on $M_+$. We have $\nabla v =0$ (cf. Theorem 3.2 in \cite[page 23]{KN}). We have to show $R(v,v)>0$ in view of \e{2.1}. Since $M_{+}=M_-$ has the same rank as $M$, there is a tangent vector $X$ in $T_{y}M_-$ such that the curvature of the 2-plane spanned by $X$ and $v(y)$ is positive. 

(III) The case of $M_{+}=M_{-}=S^{2}\cdot S^2$ in $GI$. Precisely, $M_{+}=M_-$ is obtained from $$S^{2}\times S^{2} =  {\rm(the\;  unit\;  sphere\;  in}\; \mathbb E^{3})\; \times \; {\rm the\; unit\; sphere\; in} \; \mathbb E^{3}) \subset \mathbb E^{3}\times \mathbb E^3$$ by identifying $(x,y)$ with $(-x,-y)$. 

The group $G_-$ for $M_{-}=G_{-}/K_-$ is the adjoint group but we have to take its double covering group $SO(4)$ to let it act on a neighborhood of $M_-$. The identity representation of $SO(4)$ on $\mathbb E^4$ restricts to the normal representation of $K_{-}=SO(2)\times SO(2)$ as somewhat detailed examination of the root system reveals.
Therefore, $M_-$ is unstable. Similarly for $M_+$ which is congruent with $M_-$.
\vskip.1in

\begin{remark} From the known facts about geodesics, one would not expect a simple relationship between stability and homology. More specifically, we remark that $M_+$ {\it is homologous to zero for a group manifold $M^{*}$.\/} The proof may go like this. Consider the quadratic map (a sort of Frobenius map) $f: x \mapsto s_{x}(o)$ on a symmetric space $M=G/K$ for a fixed point $o$, where $s_x$ is the symmetry at $x$. Assume $M$ is compact and orientable. Then $f$ has a nonzero degree if and only if the cohomology ring $H^{*}(M)$ is a Hopf algebra (cf. M. Clancy's thesis, University of Notre Dame, 1980). 

On the other hand, the inverse image $f^{-1}(o)$ is exactly $M_+$ and $\{o\}$. Since $H^{*}(M)$ is a Hopf algebra for a group $M^*$, it follows that every $M_+$ is homologous to zero.
\end{remark}

\section{Stability of totally real submanifolds}

A submanifold $N$ of a K\"ahlerian manifold $M$ is said to be {\it totally real} if $J(TN)$ is a subbundle of the normal bundle $T^\perp N$, where $J$ is the complex structure of $M$ (cf. for instance, \cite{CO}).

\begin{example}\label{E:5.1}{\rm Let $M$ be the complex Grassmann manifold $G^{\mathbb C}(k,n-k)$. Then the complex conjugation $c$ of ${\mathbb C}^n$ of which $\mathbb R^n$ is the fixed point set $F(c,\mathbb C^n)$ is induced on $G^{\mathbb C}(k,n-k)$, giving rise to an involutive isometry, also denote by $c$. 

$F(c,G^{\mathbb C}(k,n-k))$ is $G^{\mathbb R}(k,n-k)$, which is thus totally real and totally geodesic. More generally, if $M$ is a compact K\"ahlerian symmetric space of tube type, then the Shilov boundary of the dual domain is totally real, according to J. A. Wolf to whom we are grateful for the information.}\end{example}

In this section we make the following
\vskip.1in

{\bf Assumption 5.1.} {\it $N$ is a compact, connected, $n$-dimensional, minimal and totally real submanifold of a $2n$-dimensional K\"ahlerian manifold $M$.}
\vskip.1in

We begin with rewriting the second variation formula \e{2.1} in terms of the tangent vector field $u=Jv$ to $N$. We will show \e{2.1} is then equal to the integral \e{5.1} below, which is remarkable in that \e{5.1} does not involve the second fundamental form $B$ explicitly.

\begin{theorem}\label{T:5.1} The  minimal submanifold $N$ under Assumption 5.1 is stable if and only if
\begin{align}\label{5.1}\int_{N} \{||\nabla u ||^{2}+R^N(X,X)-R^M(X,X)\}*1\end{align}
is nonnegative for every tangent vector field $u$ on $N$, where $\nabla$ denotes the tangential connection too and $R^M$ and $R^N$ denote the Ricci forms
of $M$ and $N$, respectively.\end{theorem}
\begin{proof} Since $J$ is parallel, we have (see \cite{CO} and \cite[page 145]{KN})
\begin{align}\label{5.2} & \nabla\circ J=J\circ \nabla,
\\& \label{5.3} B\circ J=J\circ B^*\;\; {\rm on}\; \; TN,
\\&\label{5.4} K^M\circ J=J\circ K^M.\end{align}
Thus, we obtain $||\nabla v||=||\nabla Jv||=||\nabla u||$ by \e{5.2}. Fix an orthonormal basis $(e_i)$ for a tangent space to $N$, we have 
\begin{equation}\begin{aligned} \notag R(u,u):&=\sum \<K^M(e_i,v)v,e_i\>
\\ &=\sum \<K^M(e_i,Ju)Ju,e_i\>
\\ &=\sum \<K^M(Je_i,u)u,Je_i\>
\\&= R^M(u,u)-\sum \<K^M(e_i,u)u,e_i\>
\end{aligned}\end{equation}
by \e{5.4}. Thus the Gauss formula $K^M=K^N+B\wedge B^*$ yields
\begin{equation}\begin{aligned} \notag R(v,v)&= R^M(u,u)-R^N(u,u)+\sum ||B^*(e_i)u||^2\\& \hskip.4in -\sum\<B^*(e_i)e_i,B^*(u)u\>
\\&= R^M(u,u)-R^N(u,u)+||B^* u||^2,
\end{aligned}\end{equation}
since $N$ is minimal. Therefore
\begin{equation}\notag ||\nabla v||^2 -R(v,v)-||Bv||^2=||\nabla v||^2-R^M(u,u)+R^N(u,u),\end{equation}
which implies the theorem.
\end{proof}

By using Proposition 4.1 we have the following results.

\begin{theorem}\label{T:5.2} Under Assumption 5.1, we have

{\rm (1)}  If $R^M>0$ (i.e., $M$ has positive Ricci tensor) and  $H^{1}(N;{\bf R})\ne 0$, then $N$ is unstable.

{\rm (2)}  If $R^M\leq 0$ (i.e., $M$ has nonpositive Ricci tensor), then $N$ is always stable.
\end{theorem}
\begin{proof} Let $\alpha$ be the 1-form dual to a vector field $u$ tangent to $N$. Then the following formula is well-known (see, for instance, \cite[page 41]{Y})
$$\int_{N} \{||\nabla u ||^{2}+R^N(u,u)\}*1 = \int_{N} \left\{ {1\over 2} ||d\alpha ||^{2}+||\delta u ||^{2}\right\}*1,$$
where $\delta$ is the codifferential operator. Thus \e{5.1} becomes 
\begin{align}\label{5.5}\int_{N} \left\{{1\over 2}||d\alpha ||^{2}+||\delta u||^{2}-R^M(u,u)\right\}*1\end{align}
which implies (b) obviously. 

(a) follows from  formula \e{5.5} by choosing $u$ to be the vector field dual to a nonzero harmonic 1-form $\alpha$ on $N$.
\end{proof}

\begin{proposition}\label{P:5.1} Under Assumption 5.1, $N$ is stable if $N$ satisfies the condition {\rm (i)} or {\rm (ii)} below, and $N$ is unstable if $N$ satisfies condition {\rm (iii)}:

{\rm (i)}  $i^{*}R^M \leq R^N$ where i is the inclusion: $N \rightarrow M$.

{\rm (ii)}   $i^{*}R^M \leq 2R^N$ and the identity map of $N$ is stable as a harmonic map.

{\rm (iii)} $i^{*}R^M > 2R^N$ and $N$ admits a nonzero Killing vector field.
\end{proposition}
\begin{proof} Stability follows from (i) immediately by Theorem \ref{T:5.1} and  from  (ii) by the theorem  and the fact that the second variation for the identity map is (cf. \cite{Sm})
$$\int_{N} \{||{\nabla u} ||^{2}-R^N(u,u)\}*1. $$  

Sufficiency of (III) follows from Theorem \ref{T:5.1} and the formula $$\int_{N} \{||\nabla u ||^{2}-R^N(u,u)\}*1=0$$ for a Killing vector field $u$. 
\end{proof}

Finally, we apply our method in section 2 to Example \ref{E:5.1}, although the condition (iii) in Proposition \ref{P:5.1} may be verified.

 \vskip.1in
\begin{proposition}\label{P:5.2} The minimal totally real totally geodesic submanifold $G^{\mathbb R}(p,q)$ is unstable in $G^{\mathbb C}(p,q)$.\end{proposition}
\begin{proof} Let $N= G_{N}/K_{N}$ be a  totally real and totally geodesic submanifold of a compact K\"ahlerian symmetric space $M=G_{M}/K_M$. Then $N$ will be unstable if we find $c(V)<c(P')$ as in Theorem \ref{T:2.1}. 

For each simple $\mathcal P'$ in ${\mathcal P}$, there is a simple ${\it g}_{_{N}}$-module $V$ in $\Gamma (E')$ whose
members are normal vector fields $v = Ju$ for some Killing vector field $u$ in ${\it g}_{_{N}}$. This is obvious from the definition of a totally real submanifold. In the case of $G^{\mathbb R}(p,q)$ in $G^{\mathbb C}(p,q)$, ${\mathcal P}$ is simple and $c(P)=c(2{\tilde
\omega}_{1})>c({\tilde \omega}_{2}) =c({\it g}_{_{N}}),$ where ${\tilde \omega}_{2}$ denotes the highest weight in Bourbaki's notation (see \cite{B}) and ${\tilde \omega}_{1}$ is the only representation that has a smaller eigenvalue than ${\tilde \omega}_{2}$.  \end{proof}

\begin{remark} {\rm For the case $p=1$, this proposition is a special case of a theoem of Lawson-Simons  \cite{LS}.}\end{remark}

\begin{remark} {\rm  We thank H. Naitoh for pointing our the misprint $SU(2)\times SU(2)$ in Table VIII of  \cite{CN2}  for $SU(3)$, which shall read $S(U(2)\times U(1))$.}\end{remark}

\section{Appendix}

\begin{remark} The first five sections above form an unpublished article written by B.-Y. Chen, P.-F. Leung and T. Nagano in 1980 under the same title. This unpublished article has been  cited or the results of this unpublished article were cited in the books and papers \cite{C1}--\cite{T2}.
\end{remark}

\begin{remark} Let $f:N\to M$ be a compact Lagrangian minimal submanifold of a K\"ahlerian manifold. It follows immediately from the proof of Theorem \ref{T:5.2} that the index $i(f)$ of $f$ satisfies 
\begin{align} i(f)\geq \beta_1(N),\end{align}
where $\beta_1(N)$ denotes the first Betti number  of  $N$.
\end{remark}

\begin{remark}
\vskip.1in
 A reformulation of our method was given by Y. Ohnita in \cite{O}, in which
Ohnita improved our algorithm to include the formulas for the index, the
nullity and the Killing nullity of a compact totally geodesic submanifold in a compact
symmetric space.

Y. Ohnita's formulas for the index $i(f)$, the nullity $n(f)$, and the Killing nullity $n_k(f)$ are given respectively by
\vskip.1in

(a) $i(f)=\sum_{i=1}^t\sum_{\lambda\in D(G), a_\lambda <a_i} m(\lambda)d_\lambda$,
\vskip.1in

(b) $n(f)=\sum_{i=1}^t\sum_{\lambda\in D(G), a_\lambda =a_i} m(\lambda)d_\lambda$,
\vskip.1in

(c) $n_k(f)=\sum_{i=1,\mathfrak m_i^\perp\ne \{0\}}^t\dim \mathfrak g_i^\perp$,
\vskip.1in

\noindent where  $m(\lambda)=\dim\, \hbox{Hom}_K(V_\lambda,(\mathfrak m_i^\perp)^{\mathbb C})$, $d_\lambda$ is the dimension of the representation $\lambda$, and $f:N\to M$ is the totally geodesic imbedding.

\end{remark}


\begin{thebibliography}{10}

\bibitem{B} N. Bourbaki, {Groupes et algebres de Lie,} Chapters 7 and 8, Hermann, 1975. 

\bibitem{CO} B.-Y. Chen and K. Ogiue, {On totally real submanifolds,} {Trans. Amer. Math. Soc.}, {\bf 193} (1974), 257--266.

\bibitem{CN1} B.-Y. Chen and T. Nagano, {Totally geodesic submanifolds of symmetric spaces}. I, {Duke Math. J.}  {\bf 44} (1977), 745--755.

\bibitem{CN2} B.-Y. Chen and T. Nagano, {Totally geodesic submanifolds of symmetric spaces}. II, {Duke Math. J.}  {\bf 45} (1978), 405--425.

\bibitem{Fo}  A. T. Fomenko, {Minimal compacta in Riemannian manifolds, and a conjecture of Reifenberg}, Math. USSR Izv. {\bf 6} (1972), 1037-1066. (Izv. Akad. Nauk SSSR Ser. Mat. {\bf 36} (1972), 1049--1079.)

\bibitem{F} H. Federer, {Some theorems in integrabke currents}, Trans. Amer. Math. Soc. {\bf 117} (1965), 43--67.

\bibitem{H} S. Helgason, {Differential geometry, Lie groups and symmetric spaces}, Academic Press, New York, 1978.

\bibitem{KN} S. Kobayashi and K. Nomizu, Foundations of differential geometry, Vol. II, 
 John Wiley \& Sons, Inc., New York-London-Sydney, 1969.

\bibitem{LS}  H. B. Lawson and J. Simons, On stable currents and their applications to global problems in real and complex geometry,  Ann. of Math., {\bf 98} (1973), 427-450.

\bibitem{N} T. Nagano, {Stability of harmonic maps between symmetric spaces}, Harmonic maps (New Orleans, La., 1980), Lecture Notes in Math. {\bf 949} (1982), 130--137, Spring-Verlag.

\bibitem{S} J. Simons, {Minimal varieties in Riemannian manifolds}, {Ann. of Math.,} {\bf 88} (1968), 62--105.

\bibitem{Sm}  R. T. Smith, {The second variation formula for harmonic mappings}, Proc. Amer. Math. Soc. {\bf 47} (1975), 229--236.

\bibitem{T} M. Takeuchi, {Stability of certain minimal submanifolds of compact Hermitian symmetric spaces}, Tohoku Math. J., {\bf 36} (1984), 293-314.

\bibitem{Y} K. Yano, {Integral formulas in Riemannian geometry}, M. Dekker, New York, NY, 1970.

\bibitem{C1}  B.-Y. Chen, Geometry of Submanifolds and Its Applications, Science University of Tokyo, Japan, 1981.

\bibitem{C2}  B.-Y. Chen, Geometry of Slant Submanifolds, Katholieke Universiteit te Leuven, Belgium, 1990.

\bibitem{C3}  B.-Y. Chen,  Riemannian submanifolds, Handbook of differential geometry, Vol. I, 187--418, North-Holland, Amsterdam, 2000.

\bibitem{KT}  T. Kimura and M. S. Tanaka, Stability of certain minimal submanifolds in compact symmetric spaces of rank two, Differential Geom. Appl. {\bf 27} (2009), no. 1, 23--33. 

\bibitem{Ma}  K. Mashimo, On the stability of Cartan embeddings of compact symmetric spaces, Arch. Math. (Basel) {\bf 58} (1992), no. 5, 500--508. 

\bibitem{MT}  K. Mashimo and H. Tasaki,  Stability of maximal tori in compact Lie groups,  Algebras Groups Geom. {\bf 7} (1990), no. 2, 114--126.

\bibitem{MT2}  K. Mashimo and H. Tasaki, Stability of closed Lie subgroups in compact Lie groups, Kodai Math. J. {\bf 13} (1990), no. 2, 181--203.

\bibitem{Oh} Y.-G. Oh,  Second variation and stabilities of minimal Lagrangian submanifolds in K\"ahler manifolds,  Invent. Math. {\bf 101} (1990), no. 2, 501--519.

\bibitem{O}  Y. Ohita, On stability of minimal submanifolds in compact symmetric spaces,  Compositio Math. {\bf 64} (1987), no. 2, 157--189.

\bibitem{T2}  M. Takeuchi, Stability of certain minimal submanifolds of compact Hermitian symmetric spaces, Tohoku Math. J.  {\bf 36} (1984), no. 2, 293--314.


\end{thebibliography}
\end{document}